\documentclass[12pt]{amsart}

\usepackage{graphicx,amsmath,mathtools,amssymb}
\usepackage{epsfig,color,hyperref,fancyhdr,geometry}
\usepackage{extarrows}
\usepackage{dsfont}
\usepackage{epstopdf}
\usepackage{tikz-cd}
\usepackage{color}
\usepackage[numbers]{natbib}
\newtheorem{theorem}{Theorem}[section]
\newtheorem{lemma}[theorem]{Lemma}
\newtheorem{definition}[theorem]{Definition}
\newtheorem{example}[theorem]{Example}
\newtheorem{proposition}[theorem]{Proposition}
\newtheorem{remark}[theorem]{Remark}
\newtheorem{corollary}[theorem]{Corollary}
\newtheorem{conjecture}[theorem]{Conjecture}

\numberwithin{equation}{section}
\numberwithin{figure}{section}

\title{Homology of Yang-Baxter modules}

\author{Yin Tian}
\address{School of Mathematical Sciences, Beijing Normal University; 
Laboratory of Mathematics and Complex Systems, Ministry of Education, Beijing 100875, China}
\email{{\rm yintian@bnu.edu.cn}}

\author{Xiao Wang}
\address{School of Mathematics, Jilin University, Changchun, China}
\email{{\rm wangxiaotop@jlu.edu.cn}}

\author{Yuxin Zhang}
\address{School of Mathematics, Jilin University, Changchun, China}
\email{{\rm yuxinz314@gmail.com}}

\begin{document}

\begin{abstract}
We study the Yang-Baxter operator for the vector representation $V_m$ of the quantum group $U_q(sl_m)$. We consider the one-term Yang-Baxter homology with coefficients in $V_m$-modules and provide a direct sum decomposition of the one term Yang-Baxter chain complex. The homology is explicitly computed for some specific $V_m$-modules.

\end{abstract}



\maketitle

\tableofcontents

\section{Introduction}
Yang-Baxter equation is a key object both in mathematics and physics. It was discovered by C. N. Yang\cite{Yan} and R. J. Baxter\cite{Bax} independently. Solutions to the Yang-Baxter equation are deeply related with quantum groups, and can be used to construct invariants of links such as Jones polynomial\cite{Jon} and of manifolds such as Rushtikin-Turaev invariants\cite{RT}. For any Yang-Baxter operator, there is a homology theory, which generalizes the quandle homology theory and the latter has been proved useful in knot theory. In this article, we study the one-term homology theory of Yang-Baxter operators, in particular, focusing on the ones $R_{m}$ coming from quantum groups $U_{q}(sl_m)$. We expect our study can help understanding better the two-term Yang-Baxter homology, which is promising in constructing new invariants for links and manifolds.  Besides, during our work, we have found interesting algebraic structures coming from the Yang-Baxter operators themselves.  We hope our results have reflections on the theory of quantum groups as well.

Given a Yang-Baxter operator $(V,R)$, there is a unital assoiciative algebra $F(V)$. 
The notion of $V$-modules(see Definition \ref{Mod-V}) coincides with that of $F(V)$-modules. 
For each $V$-module $M$, we have a one-term Yang-Baxter chain complex $C(M)$:
$$\cdots \to M \otimes V^{\otimes n} \to M \otimes V^{\otimes n-1} \to \cdots \to M \otimes V \to M,$$
where the differential is determined by the Yang-Baxter operator $R$. 

In this paper, we focus on the Yang-Baxter operators $(V_m, R_m)$, where $V_m$ is the vector representation of 
$U_{q}(sl_m)$. The groundfield is $\mathbb{K}=\mathbb{C}(y)$, where $y$ is the quantum parameter.
The associative algebra $F(V_m)$ is the polynomial algebra in $m$ variables.   
Our main result is an explicit computation of the one-term Yang-Baxter homology. 

\begin{theorem} \label{thm main}
For the Yang-Baxter operator $(V_m, R_m)$, and any $V_m$-module $M$, we have the followings. 
\begin{enumerate}
\item The one-term Yang-Baxter complex $C(M)$ is isomorphic to a tensor product $C^f(M) \otimes B(V_m)$, 
where $C^f(M)$ is a finite complex of length $m$, and $B(V_m)$ is a graded vector space only depending on $(V_m,R_m)$.
\item The finite complex $C^f(M)$ is isomorphic to $M\otimes_{F(V_{m})} K$, where $K$ is the Koszul resolution of the one dimensional left  $F(V_m)$-module $\mathbb{K}$. 
\item The graded vector space $B(V_m)$ is a free algebra generated by $b_i$ elements of degree $i$, for $2 \le i \le m+1$, where $1-\sum\limits_{i=2}^{m+1}b_iq^i=(1-mq)(1+q)^m$. 
\end{enumerate}
\end{theorem}

{\em Further discussions:}

The generators of $B(V_m)$ are related to some canonical basis of $V_m^{\otimes n}$; It is interesting to discuss the two-term Yang-Baxter homology using similar ideas in this paper; It is natural to consider Yang-Baxter operators for other representations of $U_{q}(sl_m)$.

The paper is organised as follows. In Section 2, we define the notion of $V$-modules and the associative algebra $F(V)$. Then, we show the equivalence between $V$-modules and $F(V)$-modules. We also recall the definition of Yang-Baxter equation and that of the one-term Yang-Baxter homology of $(V,R)$ with coefficients in a $V$-module $M$.
In Section 3, we define an operator $\sigma_{n}$ that depends only on the Yang-Baxter operators $(V,R)$ themselves and study its properties in the case of $(V_m,R_m)$. Through $\sigma_n$, we obtain an eigenspace decomposition of $V_{m}^{\otimes n}$, which leads to a decomposition of the one-term Yang-Baxter chain complex of $(V_m,R_m)$. Furthermore, the above-mentioned $B(V_m):=\oplus_{n}\ker\sigma_{n}$ has a graded algebraic structure, and we know the dimension of each $\ker\sigma_{n}$. For $m=2,3$, we explicitly provide generators. We also compute the one-term Yang-Baxter homology of $R_{m}$ with coefficients in $F(V_{m})$ and certain finite dimensional $V_{m}$-modules.

\section{Preliminary}
\begin{definition}\label{Definition 1.1}
Let $k$ be a commutative ring and $V$ be a $k$-module.  If a $k$-linear map, $R:$ $V\otimes V \to V\otimes V$, satisfies the following equation\\
$$(R\otimes \mathrm{id}_{V})\circ (\mathrm{id}_{V}\otimes R)\circ (R\otimes \mathrm{id}_{V})=(\mathrm{id}_{V}\otimes R)\circ (R\otimes \mathrm{id}_{V})\circ (\mathrm{id}_{V}\otimes R),$$\ \\
then we say $(V,R)$ is a pre-Yang-Baxter operator. The equation above is called a Yang-Baxter equation.  If, in addition, $R$ is invertible, then we say $(V,R)$ is a Yang-Baxter operator. 
\end{definition}

\begin{remark}

In this article, we usually abbreviate the notation $(V,R)$ simply by $R$ when $V$ is clear from context.
    
\end{remark}

The following is the main example of Yang-Baxter operators we consider in this article.

\begin{example}\label{Example}\cite{PW2}
Let $\mathbb{K}=\mathbb{C}(y)$ and $V_{m}=\mathbb{K}\{v_{1},\cdots v_{m}\}$ be an $m$ dimension $\mathbb{K}$-space. We give a family of Yang-Baxter operators
    \[
    \begin{aligned}
         R_{m}:V_{m}\otimes V_{m}&\to V_{m}\otimes V_{m} \\
         v_{i}\otimes v_{j} &\mapsto (1-y^{2})v_{i}\otimes v_{j}+y^{2}v_{j}\otimes v_{i}\\
         v_{i}\otimes v_{i}&\mapsto v_{i}\otimes v_{i}\\
         v_{j}\otimes v_{i}&\mapsto v_{i}\otimes v_{j}
    \end{aligned}
    \]
    where $1\le i<j\le m$. For example, when $m=2$, it is 
$$
 \left[
 \begin{matrix}
   1 & 0 & 0 & 0 \\
   0 & 1-y^{2} & 1 & 0 \\
   0 & y^{2} & 0 & 0 \\
   0 & 0 & 0 & 1
  \end{matrix}
  \right]
$$
\end{example}

\begin{remark}

This family of Yang-Baxter operators leads to $sl_m$ quantum invariants following the same way as in \cite{Tur}, see \cite{PW,PW2}.
    
\end{remark}  

The next definition is about a condition naturally required for later defining the homology of Yang-Baxter operators.

\begin{definition}\label{walld}

Consider a linear map $R_{M} :M\otimes V \to M,$ such that $R_{M}\circ (R_{M}\otimes {\rm id}_{V})=R_{M}\circ (R_{M}\otimes {\rm id}_{V})\circ ({\rm id}_{M}\otimes R)$ as shown graphically in Figure \ref{wall}, we call this the wall condition.\ \\

\begin{figure}[htb]
\centerline{\psfig{file=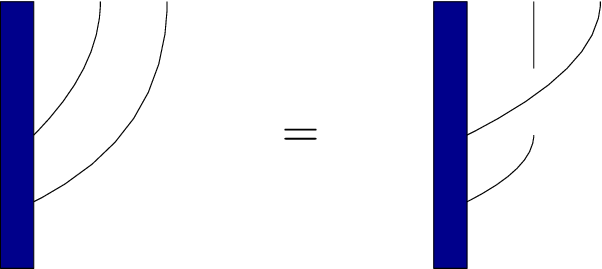,width=3.0in}}
\vspace*{8pt}
\caption{the wall condition.}\label{wall}
\end{figure}
    
\end{definition}

\begin{definition}\label{Mod-V}

Let $R$ be a Yang-Baxter operator, where $V$ is a $k$-module. We define a category associated to it, named Mod-$V$. 
\begin{enumerate}

\item Ob(Mod-$V$) consists of $k$-modules $M$ together with a $k$-module morphism $R_{M}:M\otimes V \rightarrow M$ satisfying the wall condition in Definition\ref{walld}, and we call the objects $V$-modules of the Yang-Baxter operator $R$.  When the operator $R$ is fixed, we abbreviate it as a $V$-module.

\item Mor($M$,$N$) of $V$-modules $M$ and $N$ consists of $k$-module morphisms $f:M\rightarrow N$ preserving the wall condition, i.e. $$f(R_{M}(m\otimes v))=R_{N}(f(m)\otimes v)$$
we call such $f$ a $V$-module morphism of the Yang-Baxter operator $R$.  When the operator $R$ is fixed, we abbreviate it as a $V$-module morphism.

\end{enumerate}

\end{definition}

Actually, the category just defined is equivalent to the category of modules over certain unital associative algebra we are going to define in the following. 

\begin{definition}

Let $R$ be a Yang-Baxter operator. We define an algebra $F(V)$ to be the quotient algebra of the tensor algebra $TV$ by the ideal generated by $\mathrm{im}({\rm id}_{V^{\otimes 2}}-R)$, i.e. $F(V):=TV/(\mathrm{im}({\rm id}_{V^{\otimes 2}}-R))$. For the case of set-theoretic Yang-Baxter operators, this is the same as the structure monoid(group), see \cite{Joy, Sol}.
    
\end{definition}
\begin{example}
    When we take $V=V_{m}$, ${\rm im}\,({\rm id}_{V^{\otimes 2}}-R_{m})$ is a $\mathbb{K}$-vector space spanned by a set of basis $\{v_{i}\otimes v_{j}-v_{j}\otimes v_{i}\mid {1\le i<j\le m}\}$. Thus $F(V_m)$ is the polynomial algebra in $m $ variables over $\mathbb{K}$, i.e. $$F(V_{m}):=TV/(v_{i}\otimes v_{j}-v_{j}\otimes v_{i})=\mathbb{K}[v_{1},\cdots,v_{m}].$$
\end{example}
\begin{lemma}\label{equiv}

For a given Yang-Baxter operator $(V,R)$. The $V$-module category Mod-$V$ is equivalent to the category of right $F(V)$-modules $Mod_{F(V)}$.
\[
\begin{tikzcd}
    |[draw]| \text{modules of the Yang-Baxter operator $(V,R)$} \arrow[r,leftrightarrow] & |[draw]| \text{modules of the associative algebra $F(V)$}
\end{tikzcd}
\]

\end{lemma}

\begin{proof}
For $V$-module $(M,R_{M}:M\otimes V\rightarrow M)$, we can equip it with a right $F(V)$-module structure $(M,\mu_{M}:M\otimes F(V)\rightarrow M)$ by defining $\mu_{M}|_{M\otimes V}=R_{M}$(here we identify $M\otimes V$ as its embedding in $M\otimes F(V)$). Its well-definedness follows from the fact that $R_{M}$ satisfies the wall condition. 

Conversely, we can provide $F(V)$-module $(M,\mu_{M})$ with a $V$-module structure $(M,R_{M})$ by defining $R_{M}=\mu_{M}|_{M\otimes V}$. And $R_{M}$ satisfies wall condition since $[v_{i}\otimes v_{j}]=[R(v_{i}\otimes v_{j})]$ in $F(V)$. Now, it is routine to check that a map $f:M\rightarrow N$ is a $V$-module morphism if and only if it is an $F(V)$-module morphism. 
Therefore, for $V$-modules $M$ and $N$,\[
\mathcal{G}: Mod\text{-}V \to {Mod}_{F(V)}
\]
\[
\begin{tikzcd}
(M, R_M) \arrow[r, mapsto] \arrow[d, "f"] & (M, \mu_M) \arrow[d, "\mathcal{G}(f):=f"] \\
(N, R_N) \arrow[r, mapsto] & (N, \mu_N)
\end{tikzcd}
\]

is the equivalence.
\end{proof}
\begin{remark}\
\begin{enumerate}
    \item $F(V)$ is right $F(V)$-module, so we can regard it as $V$-module as above.
    \item We use the notation $mv$ for both $R_{M}(m\otimes v)$ and $\mu_{M}(m\otimes v)$ whenever it is clear from  context. 
\end{enumerate}
   
\end{remark}
\begin{definition}\label{Definition 4.2}{\rm (\cite{Leb-1,Prz-2})}
Let $M$ be a $V$-module, and $C^{YB}_{n}=M\otimes V^{\otimes n}$. We define the face map by $d_{1,n}:=R_{M}\otimes \mathrm{id}_{V^{\otimes n-1}}$ and $d_{i,n}:C^{YB}_{n}\to C^{YB}_{n-1}$, $1<i\leq n$ by 
$$d_{i,n}:=(R_{M}\otimes \mathrm{id}_{V^{\otimes n-1}})\circ (\mathrm{id}_{M}\otimes R \otimes \mathrm{id}_{V^{\otimes n-2}})\circ \cdots \circ (\mathrm{id}_{M}\otimes \mathrm{id}_{V^{\otimes i-2}}\otimes R\otimes \mathrm{id}_{V^{\otimes n-i}}).$$
An interpretation the face maps is shown in Figure \ref{face map}.

\begin{figure}[h!]
\centerline{\psfig{file=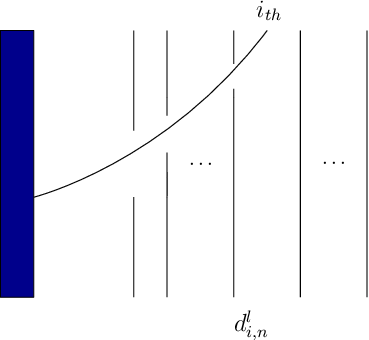,width=2.2in}}
\vspace*{8pt}
\caption{face map $d_{i,n}$.}\label{face map}
\end{figure}
 We call $C^{YB}(M):=(C^{YB}_{n},\partial^{YB}_{n})$ the one-term pre-Yang-Baxter chain complex with coefficients in the $V$-module $M$, where $\partial^{YB}_{n}=\sum_{i=1}^{n}(-1)^{i-1}d_{i,n}$.  Its homology is called the one-term Yang-Baxter homology with coefficients in the $V$-module $M$.
\end{definition}

\begin{lemma}\label{SES}

If $0\rightarrow M_1\rightarrow M_2\rightarrow M_3\rightarrow 0$ is a short exact sequence of $V$-modules, we have the corresponding long exact sequence of their Yang-Baxter homology. 
    
\end{lemma}

\begin{remark}

By Lemma \ref{equiv}, one has a short exact sequence of $V$-modules if and only if one has a short exact sequence of $F(V)$-modules. One can define Tor and Ext functors on Mod-$V$ induced from those of the category of $F(V)$-modules. 
    
\end{remark}

\section{One term Yang-Baxter homology of $sl_{m}$ operators}
In this section, we introduce the operator $\sigma_{n}$ defined in Definition \ref{sigD} and study its properties. It is an important object when computing the one-term Yang-Baxter homology with coefficients in $V_{m}$-modules.

\subsection{One eigenspace decomposition from $\sigma_{n}$}
\begin{definition}\label{sigD}
  
   Let $d^{n}_{k}$ be a linear map from $V^{\otimes n}$ to $V^{\otimes n}$ defined as below \footnote{$d_{k}^{n}$ is different from $d_{k,n}$ defined in Definition \ref{Definition 4.2}. Actually, $d_{k,n}=(R_{M}\otimes {\rm id}_{V^{\otimes n-1}})\circ d_{k}^{n}$. }{\rm :} 
   
   When $k=1$, $d_{1}^{n}:=\mathrm{id}_{V^{\otimes n}}$ and when $k>1$
   $$d^{n}_{k}:=( R \otimes \mathrm{id}_{V^{\otimes n-2}})\circ (\mathrm{id}_{V}\otimes R\otimes \mathrm{id}_{V^{\otimes n-3}})\circ \cdots \circ (\mathrm{id}_{V^{\otimes k-2}}\otimes R\otimes \mathrm{id}_{V^{\otimes n-k}}).$$ 
   
   We define a linear map $\sigma_{n}:V^{\otimes n}\to V^{\otimes n}$ by 
   $\sigma_{n}=\sum_{i=1}^{n}(-1)^{i-1}d^{n}_{i}.$
   A graphical illustration of them is shown in Figure \ref{sig} \rm{:}
    \begin{figure}[htb]
        \includegraphics[width=0.9\linewidth]{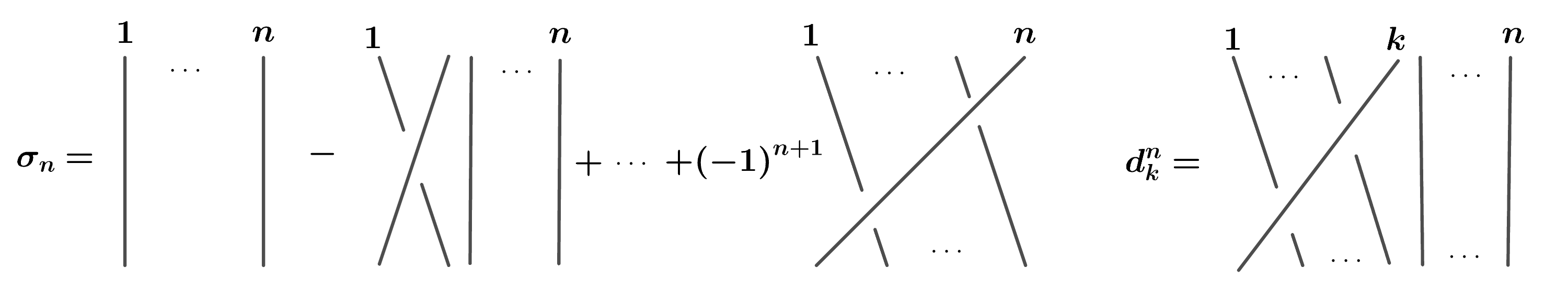}
        \caption{$\sigma_{n}$ and ${d_{k}^{n}}$}\label{sig}
    \end{figure}

\end{definition}

\begin{remark}

We note that $\partial^{YB}_{n}
=(R_{M}\otimes \mathrm{id}_{V^{\otimes  n-1}})\circ (\mathrm{id}_{M}\otimes \sigma_{n})$.
    
\end{remark}

Here we provide some properties of $\sigma_{n}$ for general Yang-Baxter operators.

\begin{lemma}\label{sigL}The following identities hold for any Yang-Baxter operator $R$:

\begin{enumerate}
    \item 
      $ \sigma_{n}=\sigma_{k}\otimes \mathrm{id}_{V^{\otimes n-k}} +(-1)^{k} ({d_{k+1}^{n} }\otimes \mathrm{id}_{V^{\otimes n-k-1}}) \circ (\mathrm{id}_{V^{\otimes n-k}}\otimes \sigma_{n-k}).$
\item $\sigma_{n}\mid_{V^{\otimes k}\otimes \ker \sigma_{n-k}}=\sigma_{k}\otimes \mathrm{id}_{V^{\otimes n-k}}\  and\  \sigma_{n}({V^{\otimes k}\otimes \ker \sigma_{n-k}} )\subset {V^{\otimes k}\otimes \ker \sigma_{n-k}}.$
\item $\forall i\le j,\  (\mathrm{id}_{V}\otimes d_{j}^{n-1})\circ d_{i}^{n}-(\mathrm{id}_{V}\otimes d_{i}^{n-1})\circ d_{j+1}^{n}=(\sigma_{2}\otimes \mathrm{id}_{V^{\otimes n-2}})\circ (\mathrm{id}_{V}\otimes d_{j}^{n-1})\circ d_{i}^{n}.$
\item $(\mathrm{id}_{V}\otimes \sigma_{n-1})\circ \sigma_{n}=(\sigma_{2}\otimes \mathrm{id}_{V^{\otimes n-2}})\circ \psi_{n}$, where $$\psi_{n}=\sum_{i\le j}(-1)^{i+j} (\mathrm{id}_{V}\otimes d_{j}^{n-1})\circ d_{i}^{n}.$$ 
\end{enumerate}
 
\end{lemma}

\begin{proof}
\begin{enumerate}
    \item[(1)(2)] Follows directly from the definition of $\sigma_{n}$
    \item[(3)] This is true since $R$ is an Yang-Baxter operator and the Reidemeister move III can be applied on the diagram representing $d_{k}^{n}$ as in Figure \ref{sig}. 
        \item [(4)] By the definition of $ \sigma_{n}$ and (3), we have
    $$\begin{aligned}
            (\mathrm{id}_{V}\otimes \sigma_{n-1})\circ \sigma_{n}&= \sum_{j=1}^{n-1} \sum_{i=1}^{n} (-1)^{i+j}(\mathrm{id}_{V}\otimes d_{j}^{n-1})\circ d_{i}^{n}\\ &=\sum_{i\le j}(-1)^{i+j} ((\mathrm{id}_{V}\otimes d_{j}^{n-1})\circ d_{i}^{n}-(\mathrm{id}_{V}\otimes d_{i}^{n-1})\circ d_{j+1}^{n})\\ &=(\sigma_{2}\otimes \mathrm{id}_{V^{\otimes n-2}})\circ \sum_{i\le j}(-1)^{i+j} (\mathrm{id}_{V}\otimes d_{j}^{n-1})\circ d_{i}^{n}.
    \end{aligned}$$
\end{enumerate}
\end{proof}

\begin{remark}
    Now we fix our Yang-Baxter operator as $R_{m}$. For the purpose to use less subscriptions, we abbreviate $V_{m}$ as $V$. We also write $F$ for $F(V_{m})$ which is actually polynomial algebra in $m$ variables over $\mathbb{K}$.
\end{remark}
For later convenience, we introduce the following notation.

\begin{definition}
    Consider the finite dimensional vector space $V$ with basis $\{v_1,...,v_m\}$.For each $n$, the $n$-bracket on $V^{\otimes n}$ is a linear map defined on standard tensor basis, $$\begin{aligned}
    [\ ,\,]_{n}: V^{\otimes n} \longrightarrow &V^{\otimes n} \\
    v_{i_{1}}\otimes \cdots \otimes v_{i_{n}} \mapsto [v_{i_{1}},\cdots ,v_{i_{n}}]_{n}&:=\sum_{\tau \in S_{n}} (-1)^{{\rm sgn} (\tau)}v_{i_{\tau_{1}}}\otimes v_{_{\tau_{2}}}\otimes \cdots \otimes v_{i_{\tau_{n}}}.\end{aligned}
    $$
    We denote $[V]_{n}$ to be the image of $V^{\otimes n}$ under $[\ ,\, ]_{n}$. It is clear that $[V]_{n}=\{0\}$ when $n>\dim(V)$. 
     
\end{definition}

\begin{remark}

From now on, when $n$ is clear in the context, we will use the notation $ [v_{i_{1}},\cdots ,v_{i_{n}}]$ instead of $ [v_{i_{1}},\cdots ,v_{i_{n}}]_{n}$.
    
\end{remark}
    
\begin{lemma}\label{prop4.4}\cite{AI}
    The $n$-bracket $[\ ,\,]_{n}$ is antisymmetric and 
    \[
\begin{aligned}
    [v_{i_{1}},\cdots ,v_{i_{n}}]&=\sum_{s=1}^{n}(-1)^{s+1}v_{i_{s}}\otimes[v_{i_{1}},\cdots, \hat{v}_{i_{s}},\cdots ,v_{i_{n}}]
      \\ &=\sum_{1\le s<t\le n}(-1)^{s+t+1}[v_{i_{s}},v_{i_{t}}]\otimes[v_{i_{1}},\cdots,\hat{v}_{i_{s}},\cdots,\hat{v}_{i_{t}},\cdots,v_{i_{n}}]
\end{aligned}
    \]
    where $\hat{v}_{i_{s}}$ indicates that $v_{i_{s}}$ is deleted.
\end{lemma}
\begin{proof}
    The first equality and the fact that the bracket is antisymmetric are from \cite{AI}. The second equality follows from a directly computation by applying the first equality twice and regrouping terms: 
    \[
    \begin{aligned}
    [v_{i_{1}},\cdots ,v_{i_{n}}]&=\sum_{s=1}^{n}(-1)^{s+1}v_{i_{s}}\otimes[v_{i_{1}},\cdots, \hat{v}_{i_{s}},\cdots ,v_{i_{n}}]
    \\&=\sum_{s=1}^{n}\sum_{s<t}^{n}(-1)^{s+t+1}(v_{i_{s}}\otimes v_{i_{t}}-v_{i_{t}}\otimes    v_{i_{s}})\otimes[v_{i_{1}},\cdots,\hat{v}_{i_{s}},\cdots,\hat{v}_{i_{t}},\cdots,v_{i_{n}}]
      \\ &=\sum_{1\le s<t\le n}(-1)^{s+t+1}[v_{i_{s}},v_{i_{t}}]\otimes[v_{i_{1}},\cdots,\hat{v}_{i_{s}},\cdots,\hat{v}_{i_{t}},\cdots,v_{i_{n}}]
\end{aligned}
    \]
\end{proof}
We provide here our main theorem and delay its proof in next subsection.

\begin{theorem}\label{main}
   There is an eigenspace  decomposition from $\sigma_{n}$\rm{:}
    \[
    V^{\otimes n}=\ker \sigma_{n}\oplus (V\otimes \ker\sigma_{n-1})\oplus ( [V]_{2}\otimes \ker\sigma_{n-2})\oplus \cdots \oplus [V]_{n}.
    \]
    where $\ker\sigma_{n}$ is the eigenspace of  $0$ and $[V]_{k}\otimes \ker\sigma_{n-k}$ is the eigenspace of the eigenvalue quantum integer $[k]_{y^{2}}$, which is $ 1+y^{2}+\cdots+y^{2k-2}$  .
\end{theorem}

This decomposition leads to a direct sum decomposition of the one-term Yang-Baxter chain complex, which proves Part $(1)$ of Theorem \ref{thm main}.

\begin{theorem}\label{complex decom} The Yang-Baxter chain complex with coefficients in the $V$-module $M$ has a direct sum decomposition into a tensor product $C^f(M) \otimes B(V_m)$, 
where $C^f(M)$ is the subcomplex $(M\otimes [V]_{n},\bar{\partial}^{YB}_{n})$ of $(C_{n}^{YB},\partial^{YB}_{n})$, where $\bar{\partial}^{YB}_{n}$ is the restriction of $\partial^{YB}_{n}$ on $M\otimes [V]_{n}$(for simplicity, we will not distinguish $\bar{\partial}^{YB}_{n}$ and $\partial^{YB}_{n}$ later on). $B(V_m)=\oplus_{n} \ker\sigma_{n}$ is a graded vector space only depending on $V_m$. Compare the diagram below:
 \[
\begin{tikzcd}[column sep=small, row sep=small]
\cdots \arrow[r] & 
M \otimes [V]_1 \otimes \ker \sigma_n \arrow[r, "\partial^{YB}_{n+1}"] \arrow[d, phantom, "\oplus" description] & 
M \otimes \ker \sigma_n \arrow[r, "\partial^{YB}_n"] \arrow[d, phantom, "\oplus" description] &
0 \arrow[d, phantom] \\
\cdots \arrow[r] & 
M \otimes [V]_2 \otimes \ker \sigma_{n-1} \arrow[r, "\partial^{YB}_{n+1}"] \arrow[d, phantom, "\oplus" description] & 
M \otimes [V]_1 \otimes \ker \sigma_{n-1} \arrow[r, "\partial^{YB}_n"] \arrow[d, phantom, "\oplus" description] & 
M \otimes \ker \sigma_{n-1} \arrow[r, "\partial^{YB}_{n-1}"] \arrow[d, phantom, "\oplus" description] & 
0 \\
& \vdots \arrow[d, phantom, "\oplus" description] & 
\vdots \arrow[d, phantom, "\oplus" description] & 
\vdots \arrow[d, phantom, "\oplus" description] \\
\cdots \arrow[r] &
M \otimes [V]_{n+1} \arrow[r, "\partial^{YB}_{n+1}"] &
M \otimes [V]_n \arrow[r, "\partial^{YB}_n"] &
M \otimes [V]_{n-1} \arrow[r, "\partial^{YB}_{n-1}"] &
\cdots
\end{tikzcd}
\]

\end{theorem}
\begin{proof}

We have$$
\begin{aligned}
    \partial^{YB}_{n}(M\otimes[V]_{k}\otimes\ker\sigma_{n-k})=&(R_{M}\otimes \mathrm{id}_{V^{\otimes n-1}})\circ (\mathrm{id}_{M}\otimes \sigma_{n})
(M\otimes[V]_{k}\otimes\ker\sigma_{n-k})
    \\= & (R_{M}\otimes \mathrm{id}_{V^{\otimes n-1}})
(M\otimes[V]_{k}\otimes\ker\sigma_{n-k})\subset M\otimes [V]_{k-1}\otimes\ker\sigma_{n-k}.
\end{aligned}$$
The second equality follows from Theorem \ref{main} and the last inclusion is given by Lemma \ref{prop4.4}.   
\end{proof}

The following proposition explains Part $(2)$ of Theorem \ref{thm main}.

\begin{theorem}
\label{Koszul finite}
The finite complex $C^f(M)$ is isomorphic to $(M\otimes_{F} \bigwedge^{n}F^{m}, {\rm id}_{M}\otimes d_{n})_{n=1}^{m}$, where $ (\bigwedge^{n}F^{m},d_{n})_{n=1}^{m}$ is the Koszul resolution of the one dimensional left  $F$-module $\mathbb{K}$.
\[
\begin{tikzcd}
    0\arrow[r]&M\otimes_{F}\bigwedge^{m}F^{m} \arrow[d,"f_{m}"]\arrow[r,"{\rm id}_{M}\otimes d_{m}"]& \cdots \arrow[r] & M\otimes_{F}\bigwedge^{2}F^{m} \arrow[d,"f_{2}"] \arrow[r,"\mathrm{id}_{M}\otimes \mathop{d}_{2}"] & M\otimes_{F} F^{m} \arrow[d,"f_{1}"] \arrow[r,"\mathrm{id}_{M}\otimes \mathop{d}_{1}"] & M\otimes_{F}F \arrow[d,"\simeq"]  \\
    0\arrow[r] & M\otimes_{\mathbb{K}}[V]_{m}\arrow[r,"\partial^{YB}_{m}"] & \cdots \arrow[r]&M\otimes_{\mathbb{K}}[V]_{2} \arrow[r,"\partial^{YB}_{2}"] &M \otimes_{\mathbb{K}}[V]_{1}\arrow[r,"\partial^{YB}_{1}"] & M 
\end{tikzcd}
\]
The chain isomorphism $f_{k}$ is defined as following:
\[
\begin{aligned}
    f_{k}:M\otimes_{F}\bigwedge^{k}F^{m} &\to M\otimes_{\mathbb{K}}[V]_{k}\\ m\otimes( e_{i_{1}}\wedge\cdots \wedge e_{i_{k}}) &\mapsto \frac{1}{[k]_{y^2}!}m\otimes[v_{i_{1}},\cdots,v_{i_{k}}].
\end{aligned}
\] 
\end{theorem}
\begin{proof}
    $f_{k}$ is a linear isomorphism clearly, so we only need to show it is a chain map.
    \[
    \begin{tikzcd}
        m\otimes (e_{i_{1}}\wedge\cdots \wedge e_{i_{k}}) \arrow[r,shorten=2mm,mapsto,"\mathrm{id}_{M}\otimes d_{k}"] \arrow[d,mapsto,"f_{k}"] & \sum_{j=1}^{k}(-1)^{j+1}{mv_{i_{j}}}\otimes(e_{i_{1}}\wedge\cdots\wedge\hat{e}_{i_{j}}\wedge\cdots\wedge e_{i_{k}}) \arrow[d,mapsto,"f_{k-1}"] \\ \frac{1}{[k]_{y^2}!}m\otimes[v_{i_{1}},\cdots,v_{i_{k}}] \arrow[r,mapsto,"\partial^{YB}_{k}"] & \sum_{j=1}^{k}(-1)^{j+1}\frac{1}{[k-1]_{y^2}!}mv_{i_{j}}\otimes [v_{i_{1}},\cdots,\hat{v}_{i_{j}},\cdots v_{i_{k}}]
    \end{tikzcd}
    \]
    Note that by Theorem \ref{main}, $[v_{i_{1}},\cdots,v_{i_{k}}]$ is an eigenvector of $\sigma_{k}$ with the eigenvalue $[k]_{y^2}$ and $\partial_{k}^{YB}=(R_{M}\otimes {\rm id}_{V^{\otimes k}})\circ({\rm id}_{M}\otimes \sigma_{k})$, we have the commutative diagram above.
\end{proof}
As an application, we compute the one-term Yang-Baxter homology of with coefficients in $F$.  
\begin{example}
    The one-term Yang-Baxter homology with coefficient in the  $V_{m}$-module $F$ is $H^{YB}_{n}(F)=\{ 1\}\otimes \ker\sigma_{n}$.
\end{example}
\begin{proof}

We have $$H^{YB}_{n}(F)=\bigoplus_{k=0}^{n}\frac{\ker(\partial_{n}^{YB}\mid_{F\otimes[V]_{k}\otimes\ker\sigma_{n-k}})}{{\rm im}\,(\partial_{n+1}^{YB}\mid_{F\otimes[V]_{k+1}\otimes\ker\sigma_{n-k}})}$$
from Theorem \ref{complex decom}. When $k=0$, we have $$\frac{\ker(\partial_{n}^{YB}\mid_{F\otimes\ker\sigma_{n}})}{{\rm im}\,(\partial_{n+1}^{YB}\mid_{F\otimes[V]_{1}\otimes\ker\sigma_{n}})}=\frac{F\otimes\ker\sigma_{n}}{R_{F}(F\otimes V)\otimes\ker\sigma_{n}}=\{1\}\otimes \ker\sigma_{n}$$ by direct calculation. Let $M=F$ in Theorem \ref{Koszul finite}, the chain complex $(F\otimes_{F}\bigwedge^{k}F^{m},{\rm id}_{F}\otimes d_{k})$ is just Koszul resolution. Thus we know 
 \[
    \begin{tikzcd}
       F\otimes [V]_{k+1} \arrow[r,"\partial^{YB}_{k+1}"] & F\otimes [V]_{k}\ \arrow[r,"\partial^{YB}_{k}"] & F\otimes [V]_{k-1}
    \end{tikzcd}
    \]
is exact, so is 
 \[
    \begin{tikzcd}
       F\otimes [V]_{k+1}\otimes \ker\sigma_{n-k} \arrow[r,"\partial^{YB}_{n+1}"] & F\otimes [V]_{k}\otimes \ker\sigma_{n-k} \arrow[r,"\partial^{YB}_{n}"] & F\otimes [V]_{k-1}\otimes \ker\sigma_{n-k-1} .
    \end{tikzcd}
    \]
    Thus we have $$\frac{\ker(\partial_{n}^{YB}\mid_{F\otimes[V]_{k}\otimes\ker\sigma_{n-k}})}{{\rm im}\,(\partial_{n+1}^{YB}\mid_{F\otimes[V]_{k+1}\otimes\ker\sigma_{n-k}})}=\{0\}$$ when $k>0$.
\end{proof}

\subsection{The proof of Theorem \ref{main}}

In this subsection, we will establish the proof of Theorem \ref{main}. For that, we will first prove several lemmas.

\begin{lemma}\label{Lemma 3.12}
    $(F\otimes [V]_{k},\delta _{k}:=(R_{F}\otimes{\rm id}_{V^{\otimes k-1}})\mid_{F\otimes[V]_{k}})$ is a chain complex and is isomorphic to the chain complex $(F\otimes_{F}\bigwedge^{k}F^{m},{\rm id}_{F}\otimes d_{k})${\rm :}
    \[
    \begin{tikzcd}
    0\arrow[r]&F\otimes_{F}\bigwedge^{m}F^{m} \arrow[d,"\bar{f}_{m}"]\arrow[r,"{\rm id}_{F}\otimes d_{m}"]& \cdots \arrow[r] & F\otimes_{F}\bigwedge^{2}F^{m} \arrow[d,"\bar{f}_{2}"] \arrow[r,"\mathrm{id}_{F}\otimes \mathop{d}_{2}"] & F\otimes_{F} F^{m} \arrow[d,"\bar{f}_{1}"] \arrow[r,"\mathrm{id}_{F}\otimes \mathop{d}_{1}"] & F\otimes_{F}F \arrow[d,"\simeq"]  \\
    0\arrow[r] & F\otimes_{\mathbb{K}}[V]_{m}\arrow[r,"\delta_{m}"] & \cdots \arrow[r]&F\otimes_{\mathbb{K}}[V]_{2} \arrow[r,"\delta_{2}"] &F \otimes_{\mathbb{K}}[V]_{1}\arrow[r,"\delta_{1}"] & F 
\end{tikzcd}
    \]
    where $\bar{f}_{k}$ is defined on the basis by
    \[
    \bar{f}_{k}(x\otimes e_{i_{1}}\wedge \cdots\wedge e_{i_{k}}):=x\otimes[v_{i_{1}},\cdots,v_{i_k}].
    \]
    Therefore, both of them are exact for $k\ge 1$.
\end{lemma}

\begin{proof}
     For $k\ge 2$, we have $[V]_{k}\subset [V]_{2}\otimes V^{\otimes k-2}={\rm im}\,({\rm id }_{V^{\otimes 2}}-R)\otimes V^{\otimes k-2}$ by Lemma \ref{prop4.4}. Thus $$\delta_{k-1}\circ\delta_{k}=(R_{F}\otimes {\rm id}_{V^{\otimes k-2}})\circ (R_{F}\otimes {\rm id}_{V^{\otimes k-1}})([V]_{k})=0.$$ It is clear that $\bar{f}_{k}$ is an isomorphism and the diagram commutes.

     We note that $H_{k}(F\otimes_{F}\bigwedge^{*}F^{m},{\rm id}_{F}\otimes d_{*})$ is $Tor_{k}^{F}(F,\mathbb{K})$ which is trivial for $k\ge 1$. Thus, the lower chain complex is also exact for $k\ge 1$.
\end{proof}

\begin{lemma}\label{main lemma}
    The following assertions are true.
    \begin{enumerate}
        \item $[V]_{n}=V\otimes [V]_{n-1}\cap [V]_{2}\otimes V^{\otimes n-2}=[V]_{2}\otimes V^{\otimes n-2} \cap V\otimes [V]_{2}\otimes V^{\otimes n-3}\cap \cdots \cap V^{\otimes n-2}\otimes [V]_{2}$.
        \item Define  $\phi_{n}^{i}:V^{\otimes n}\to V^{\otimes n}$ by $$\phi_{n}^{i}=(\mathrm{id}_{V^{\otimes n-i}}\otimes \sigma_{i})\circ (\mathrm{id}_{V^{\otimes n-i-1}}\otimes \sigma_{i+1} )\circ \cdots \circ \sigma_{n}.$$
        We have\begin{enumerate}
            \item ${\rm im}(\phi_{n}^{i})\subset [V]_{n-i+1}\otimes V^{\otimes i-1}$.
         \item $\phi_{n}^{1}$ is a surjective linear map onto $[V]_{n}$ and \[
        \phi_{n}^{1}(v_{i_{1}},\cdots,v_{i_{n}})=y^{2\, {\rm inv}(i_{n},\cdots ,i_{1})}[v_{i_{1}},\cdots,v_{i_{n}}]
        ,\]
        where ${\rm inv}(i_{n},\cdots,i_{i})$ is the inversion number of the sequence.
        \end{enumerate}
    \end{enumerate}
\end{lemma}
\begin{proof}
\begin{enumerate}
\item 
Here we denote $(R_{F}\otimes {\rm id}_{V^{\otimes n-1}})(F\otimes [V]_{n})$ by $F[V]_{n}$ for convenience.
By Lemma \ref{Lemma 3.12}, we have $\ker R_{F}=\ker\delta_{1}={\rm im}\,\delta_{2}=F[V]_{2}$ and $$F[V]_{n}={\rm im}\,\delta_{n}=\ker\delta_{n-1}=F\otimes [V]_{n-1}\cap F[V]_{2}\otimes V^{\otimes n-2}.$$ We can identify $V^{\otimes n}$ with $(R_{F}\otimes {\rm id}_{V^{\otimes n-1}})(\mathbb{K}\otimes V^{\otimes n})$ as a subspace of $F\otimes V^{\otimes n-1}$ and obtain the first equality\footnote{Here we identify both $V^{\otimes n}$ and its subspaces as subspaces of $F\otimes V^{\otimes n-1}$.}
\[
\begin{aligned}    
 \ [V]_{n}&=F[V]_{n}\cap V^{\otimes n}\\&=(F\otimes[V]_{n-1}\cap V^{\otimes n})\cap (F[V]_{2}\otimes V^{\otimes n-2}\cap V^{\otimes n})\\&=V\otimes [V]_{n-1}\cap [V]_{2}\otimes V^{\otimes n-2}.
\end{aligned}
\]
The second equality can obtain from the first one clearly.
\item \begin{enumerate}
    \item 
    Note that ${\rm im}\,\sigma_{2}={\rm im}\,({\rm id}_{V^{\otimes2}}-R)=[V]_{2}$. Applying Lemma \ref{sigL}$(4)$ to  $(\mathrm{id}_{V^{\otimes k}}\otimes \sigma_{n-k})\circ (\mathrm{id}_{V^{\otimes k-1}}\otimes \sigma_{n-k+1})$ in $\phi_{n}^{i}$, we obtain ${\rm im}\,\phi_{n}^{i}\subset V^{\otimes k-1}\otimes [V]_{2}\otimes V^{\otimes n-k-1}$. As $k$ taking values from $1$ to $n-i$, we have $$\qquad  \qquad  {\rm im}\,\phi_{n}^{i}\subset [V]_{2}\otimes V^{\otimes n-2} \cap \cdots \cap V^{\otimes n-1-i}\otimes [V]_{2}\otimes V^{\otimes i-1}=[V]_{n-i+1}\otimes V^{\otimes i-1}.$$ 
    \item
        Next we show the formula $$\phi_{n}^{1}(v_{i_{1}},\cdots,v_{i_{n}})=y^{2 \,{\rm inv}(i_{n},\cdots ,i_{1})}[v_{i_{1}},\cdots,v_{i_{n}}],$$ which implies $\phi_{n}^{1}$ is surjective. Since ${\rm im}\,\phi_{n}^{1}\subset [V]_{n}$, we know $$\phi_{n}^{1}(v_{i_{1}},\cdots,v_{i_{n}})=\alpha[v_{i_{1}},\cdots,v_{i_{n}}],$$ where $\alpha \in \mathbb{K}$. By direct computation, it is easy to see that the basis $v_{i_{n}}\otimes \cdots \otimes v_{i_{1}}$ in $\phi_{n}^{1}(v_{i_{1}},\cdots,v_{i_{n}})$ appears only once and whose coefficient is $y^{2 \,{\rm inv}(i_{n},\cdots ,i_{1})}$. Therefore, we have 
        $\alpha=y^{2 \,{\rm inv}(i_{n},\cdots ,i_{1})}$.   
\end{enumerate}
\end{enumerate}    
\end{proof}

\begin{lemma}\label{lemma of eigen} $[V]_{n}$ is a subspace of the eigenspace of $\sigma_{n}$ with the eigenvalue $[n]_{y^{2}}=1+y^{2}+\cdots+y^{2n-2}$.
\end{lemma}
\begin{proof}
Note that $[V]_{2}$ is the eigenspace of $R$ with the eigenvalue $-y^2$. And we have $$[V]_{n}=[V]_{2}\otimes V^{\otimes n-2} \cap V\otimes [V]_{2}\otimes V^{\otimes n-3}\cap \cdots \cap V^{\otimes n-2}\otimes [V]_{2}.$$
Thus, $\forall k\in \{1,\cdots,n\}$ and $\forall v\in [V]_{n}$ $$ d^{n}_{k}(v)=(-1)^{k-1}y^{2k-2}v$$
    and $$\sigma_{n}(v):=\sum_{k=1}^{n}(-1)^{k-1}d^{n}_{k}(v)=(1+y^{2}+\cdots+y^{2n-2})v.$$
\end{proof}

\begin{corollary}\label{eigen-sig}

$[V]_{k}\otimes \ker \sigma_{n-k}$ is a subspace of the eigenspace of $\sigma_{n}$ with the eigenvalue $[k]_{y^{2}}=1+y^{2}+\cdots+y^{2k-2}$.
    
\end{corollary}

\begin{lemma} We observe the following way to count the dimension of $V^{\otimes n}$:
    \[
\dim(V^{\otimes n})=\dim(\ker\sigma_{n})+ \dim (V\otimes \ker\sigma_{n-1})+ \dim([V]_{2}\otimes \ker\sigma_{n-2})+ \cdots + \dim([V]_{n})
    \]
\end{lemma}

\begin{proof}
   Consider the map $\phi_{n}^{1}=\phi_{n}^{2}=(\mathrm{id}_{V^{\otimes n-2}}\otimes \sigma_{2}) \circ  \cdots \circ\sigma_{n}$, we have $$\dim(V^{\otimes n})=\dim((\phi_{n}^{1})^{-1}({\rm im}\,\phi_{n}^{1}))=\dim({\rm im}\,\phi_{n}^{1})+\dim(\ker\phi_{n}^{1}).$$
   Then let us further decompose $\dim(\ker\phi_{n}^{2})=\dim(\ker\phi_{n}^{1})$, 
$$\dim(\ker\phi_{n}^{2})=\dim((\phi_{n}^{3})^{-1}(\ker (\mathrm{id}_{V^{\otimes n-2}}\otimes \sigma_2)))=\dim\ker\phi_{n}^{3}+\dim({\rm im}\,\phi_{n}^{3}\cap \ker (\mathrm{id}_{V^{\otimes n-2}}\otimes \sigma_2))$$
Recurvisely, we have   
   $$\dim (V^{\otimes n})=\dim({\rm im}\,\phi_{n}^{1})+\sum_{i=2}^{n}\dim ({\rm im}\,\phi_{n}^{i} \cap \ker (\mathrm{id}_{V^{\otimes n+1-i}}\otimes \sigma_{i-1}))+\dim(\ker\sigma_{n}).$$
   Next, we claim $${\rm im}\,\phi_{n}^{i}\cap \ker (\mathrm{id}_{V^{\otimes n+1-i}}\otimes \sigma_{i-1})=[V]_{n+1-i}\otimes \ker\sigma_{i-1}$$ 
   which proves Lemma directly. Firstly, ${\rm im}\,\phi_{n}^{i}\subset [V]_{n+1-i}\otimes V^{\otimes i-1}$ from Lemma \ref{main lemma}(3)(a), thus $${\rm im}\,\phi_{n}^{i}\cap \ker (\mathrm{id}_{V^{\otimes n+1-i}}\otimes \sigma_{i-1})\subset [V]_{n+1-i}\otimes \ker\sigma_{i-1}.$$
   Secondly, $\phi_{n}^{i}$ restricts to $V^{\otimes n+1-i}\otimes \ker\sigma_{i-1}$, $\phi_{n}^{i}$ will be the same as $\phi_{n+1-i}^{1}\otimes \mathrm{id}_{V^{\otimes i-1}}$ which is surjective to $[V]_{n+1-i}\otimes \ker\sigma_{i-1}$. Thus we have
   $${\rm im}\,\phi_{n}^{i}\cap \ker (\mathrm{id}_{V^{\otimes n+1-i}}\otimes \sigma_{i-1})\supset [V]_{n+1-i}\otimes \ker\sigma_{i-1}.$$
\end{proof}

Now, let us provide the proof of Theorem \ref{main}.
\begin{proof}[{\rm \textbf{Proof of Theorem \ref{main}}}]
     By Corollary \ref{eigen-sig}, $[V]_{k}\otimes \ker\sigma_{n-k}$ is a subspace of the eigenspace of the eigenvalue $[k]_{y^2}=1+y^{2}+\cdots+y^{2k-2}$. Thus 
    we have the direct sum decompostion since they all belongs to different eigenspaces, $$\ker \sigma_{n}\oplus V\otimes \ker\sigma_{n-1}\oplus [V]_{2}\otimes \ker\sigma_{n-2}\oplus \cdots \oplus [V]_{n}\subset V^{\otimes n}.$$ And we have
        \[
\dim(V^{\otimes n})=\dim(\ker\sigma_{n})+ \dim (V\otimes \ker\sigma_{n-1})+ \dim([V]_{2}\otimes \ker\sigma_{n-2})+ \cdots + \dim([V]_{n})
    \]
    Thus  \[
    V^{\otimes n}=\ker \sigma_{n}\oplus V\otimes \ker\sigma_{n-1}\oplus [V]_{2}\otimes \ker\sigma_{n-2}\oplus \cdots \oplus [V]_{n}.
    \]
\end{proof}

\subsection{The structure of the kernel of $\sigma_{n}$}
Let $m$ be a fixed integer and  $M(n)=\dim(\ker\sigma_{n})$. We define $M(0)=1$ and $M(n)=0$ when $n<0$. In this subsection, we will show that $B(V)=\bigoplus_{n}\ker\sigma_{n}$ is an algebra. In particular, a generating set of $B(V)$ is provided inductively when $m=2,3$.

First, let us determine $M(n)$. From the eigenspace decomposition in Theorem \ref{main}, we can get some formulas of $M(n)$.

\begin{lemma}\label{FomulaLemma}
    For $M(n)$, we have the following:
    \begin{enumerate}
        \item $$m^{n}=\sum_{i=0}^{m}\binom{m}{i}M(n-i)$$
        \item $$0=\sum_{i=0}^{m}(i-1)\binom{m+1}{i}M(n-i)$$
    \end{enumerate}
\end{lemma}
\begin{proof}\
\begin{enumerate}
    \item It follows from Theorem \ref{main} directly. 
    \item 
    \[
    \begin{aligned}
        0&=m\cdot m^{n-1}-m^{n}=\sum_{i=0}^{m}m\binom{m}{i}M(n-1-i)-\sum_{i=0}^{m}(m-i)M(n-i)\\
        &=-M(n)+\sum_{i=0}^{m-1}(m\binom{m}{i}-\binom{m}{i+1})M(n-1-i)+m\binom{m}{m}M(n-m-1)\\&=-M(n)+\sum_{i=0}^{m-1}i\binom{m+1}{i+1}M(n-1-i)+m\binom{m}{m}M(n-m-1)\\&=\sum_{i=0}^{m+1}(i-1)\binom{m+1}{i}M(n-i)
    \end{aligned}
    \]
\end{enumerate}
\end{proof}

\begin{proposition} \label{KernelTensor}
$(B(V),\otimes)$ is a graded algebra, i.e.
    $$\ker\sigma_{s}\otimes \ker\sigma_{t}\subset \ker\sigma_{s+t}.$$
    \end{proposition}
    \begin{proof}
    It follows directly from the Lemma \ref{sigL}(1).
    \end{proof}

\begin{definition}\label{tilde} Inductively, for each $n$, we choose a subspace ${\widetilde{\ker}\sigma_{n}}$ of $\ker\sigma_{n}$ so that 
    $$ ({\ker\sigma_{n-1}}\otimes {\widetilde{\ker}\sigma_{1}}+ \cdots + {\ker\sigma_{1}}\otimes {\widetilde{\ker}\sigma_{n-1}})\oplus{\widetilde{\ker}\sigma_{n}}=\ker\sigma_{n}$$  with initial conditions ${\widetilde{\ker}\sigma_{1}}=\ker\sigma_{1}=\{0\}$.
\end{definition}
\begin{proposition}
There is a direct sum :
\[
(V^{\otimes n-1}\otimes {\widetilde{\ker}\sigma_{1}})\oplus( V^{\otimes n-2}\otimes {\widetilde{\ker}\sigma_{2}})\oplus\cdots \oplus {\widetilde{\ker}\sigma_{n}}
\]
\end{proposition}
\begin{proof}
For $n=3$, $V\otimes {\widetilde{\ker}
\sigma_{2}}\cap {\widetilde{\ker}\sigma_{3}}=\{0\}$ since they belong to the eigenspace of $\sigma_{n}$ with eigenvalues $1$ and $0$ respectively. Now we assume that the statement is true for $n-1$, we only need to show $$((V^{\otimes n-1}\otimes {\widetilde{\ker}\sigma_{1}})\oplus (V^{\otimes n-2}\otimes {\widetilde{\ker}\sigma_{2}})\oplus\cdots \oplus (V\otimes{\widetilde{\ker}\sigma_{n-1}}))\cap {\widetilde{\ker}\sigma_{n}}=\{0\}.$$ We have 
    \[
    \begin{aligned}
        &((V^{\otimes n-1}\otimes {\widetilde{\ker}\sigma_{1}})\oplus (V^{\otimes n-2}\otimes {\widetilde{\ker}\sigma_{2}})\oplus\cdots \oplus (V\otimes{\widetilde{\ker}\sigma_{n-1}}))\cap {\ker\sigma_{n}}\\
        &{=}((V^{\otimes n-1}\otimes {\widetilde{\ker}\sigma_{1}}) \cap \ker\sigma_{n})\oplus( (V^{\otimes n-2}\otimes {\widetilde{\ker}\sigma_{2}})\cap {\ker\sigma_{n}})\oplus\cdots \oplus( (V\otimes{\widetilde{\ker}\sigma_{n-1}})\cap {\ker\sigma_{n}})\\
        &=(\ker\sigma_{n-1}\otimes {\widetilde{\ker}\sigma_{1}}) \oplus (\ker\sigma_{n-2}\otimes {\widetilde{\ker}\sigma_{2}})\oplus \cdots\oplus (\ker\sigma_{1}\otimes {\widetilde{\ker}\sigma_{n-1}}),
    \end{aligned}
    \]
    where the first equality holds by Lemma \ref{sigL}(2), which is stating $\sigma_{n}$ preserves the direct sum structure.  
    Therefore, $$((V^{\otimes n-1}\otimes {\widetilde{\ker}\sigma_{1}})\oplus (V^{\otimes n-2}\otimes {\widetilde{\ker}\sigma_{2}})\oplus\cdots \oplus (V\otimes{\widetilde{\ker}\sigma_{n-1}}))\cap {\widetilde{\ker}\sigma_{n}}=\{0\}.$$ 
\end{proof}

\begin{corollary}\label{CorollaryofDS}
The sums in Definition \ref{tilde} are actually direct sums. Therefore, $B(V)$, as an algebra, is generated by the chosen basis of $\widetilde{\ker}\sigma_{i}$. 
     $${\widetilde{\ker}\sigma_{n}}\oplus ({\ker\sigma_{n-1}}\otimes {\widetilde{\ker}\sigma_{1}})\oplus \cdots \oplus ({\ker\sigma_{1}}\otimes {\widetilde{\ker}\sigma_{n-1}})=\ker\sigma_{n}$$ 
\end{corollary}
\begin{proof}
    It is clearly since $\ker\sigma_{n-k}\otimes \widetilde{\ker}_{k}\subset V^{\otimes n-k}\otimes \widetilde{\ker}_{k}$.
\end{proof}
By dimension arguments, we see that ${\widetilde{\ker}\sigma_{n}}$ will vanish for large $n$.
\begin{theorem}\label{KernelThm}
        When $n>m+1$, ${\widetilde{\ker}\sigma_{n}}=\{0\}$ i.e.
    \[   \ker\sigma_{n}=({\ker\sigma_{n-2}}\otimes {\widetilde{\ker}\sigma_{2}})\oplus \cdots \oplus ({\ker\sigma_{n-m-1}}\otimes {\widetilde{\ker}\sigma_{m+1}})
    \]
\end{theorem}

\begin{proof}
    Firstly, it is a direct sum from the Corollary \ref{CorollaryofDS}. And we have $$({\ker\sigma_{n-2}}\otimes {\widetilde{\ker}\sigma_{2}})\oplus \cdots \oplus( {\ker\sigma_{n-m-1}}\otimes {\widetilde{\ker}\sigma_{m+1}})\subset \ker\sigma_{n}$$ from Proposition \ref{KernelTensor}. Thus we only need to show the dimensions match, for which, we now prove $\dim({\widetilde{\ker}\sigma_{i}})=(i-1)\binom{m+1}{i}$. For $i=2$, we have $${\widetilde{\ker}\sigma_{2}}=\ker\sigma_{2}=span_{\mathbb{K}}\{v_{i}\otimes v_{i}\, ,v_{i}\otimes v_{j}+y^{2}v_{j}\otimes v_{i}\mid i<j\}$$ which implies  $\dim({\widetilde{\ker}\sigma_{2}})=\binom{m}{1}+\binom{m}{2}=\binom{m+1}{2}$. Assuming that $\forall k<i\, ,\dim({\widetilde{\ker}\sigma_{k}})=(k-1)\binom{m+1}{k}$, we have the following by \ref{FomulaLemma}(2). $$\dim({\widetilde{\ker}\sigma_{i})}=M(i)-\sum_{k=1}^{i-1}\dim({\widetilde{\ker}\sigma_{k}})M(i-k)=M(i)-\sum_{k=1}^{i-1}(k-1)\binom{m+1}{k}M(i-k)=(i-1)\binom{m+1}{i}$$
    By \ref{FomulaLemma}(2) again, we have $$\dim(\ker\sigma_{n-2}\otimes {\widetilde{\ker}\sigma_{2}}\oplus \cdots \oplus \ker\sigma_{n-m-1}\otimes {\widetilde{\ker}\sigma_{m+1}})=\dim( \ker\sigma_{n}).$$
    Therefore, the claim follows.
\end{proof}

Theorem \ref{KernelThm} answers Part $(3)$ in Theorem \ref{thm main}.
\begin{theorem} \label{3.23Thm}
    The graded vector space $B(V)$ is a free algebra generated by the chosen basis of $\widetilde{\ker}\sigma_{i}$, for $2 \le i \le m+1$. Let $b_{i}=\dim\widetilde{\ker}\sigma_{i}$, we have $1-\sum\limits_{i=2}^{m+1}b_iq^i=(1-mq)(1+q)^m$. 
\end{theorem}
\begin{proof}
From Corollary \ref{CorollaryofDS}, we see that $B(V)$ is generated by $\widetilde{\ker}\sigma_{i}$ and by Theorem \ref{KernelThm}, we only need to take $2\leq i \leq m+1$. For the formula, we have
    $$b_{i}=\dim( \widetilde{\ker}\sigma_{i})=(i-1)\binom{m+1}{i}=m\binom{m}{i}-\binom{m}{i+1},$$ where $m\binom{m}{i}-\binom{m}{i+1}$ is exactly the coefficient of $q^{i}$ after expanding $(1-mq)(1+q)^m$.
\end{proof}

When $m=2,3$, we are able to provide a generating set for $\ker\sigma_{n}$ explicitly.

 \begin{proposition}
    If $x\in \ker\sigma_{n-1}$, $x=\sum_{\alpha=1}^{N} f_{\alpha}v_{\alpha_{1}}\otimes \cdots\otimes v_{\alpha_{n-1}}$ where $f_{\alpha}\in \mathbb{K}$, then for all $k\le min\{\alpha_{i}\mid1\le \alpha\le N,1\le i\le n-1\}$, we have $${\omega}_{k}(x):=v_{k}\otimes x +(-1)^{n}  x\otimes v_{k}\in \ker\sigma_{n}.$$
\end{proposition}
\begin{proof}
    We have $\sigma_{n}(v_{i}\otimes x+(-1)^{n}x\otimes v_{i})=v_{i}\otimes x-v_{i}\otimes x=0$ by direct calculation.
\end{proof}

As corollaries of Theorem \ref{3.23Thm}, we provide the following examples.

\begin{example}
    
\label{m=2}
    For $m=2$, $\ker\sigma_{n}$ is generated by ${\widetilde{\ker}\sigma_{2}}$ and ${\widetilde{\ker}\sigma_{3}}$, where 
\[
\begin{aligned}
 &{\widetilde{\ker}\sigma_{2}}=span_{\mathbb{K}}\{v_{1}\otimes v_{1},v_{2}\otimes v_{2},v_{1}\otimes v_{2}+y^{2}v_{2}\otimes v_{1}\},\\
 &{\widetilde{\ker}\sigma_{3}}=span_{\mathbb{K}}\{\omega_{1}(v_{2}\otimes v_{2}), \omega_{1}(v_{1}\otimes v_{2}+y^2v_{2}\otimes v_{1})
 \}   .
\end{aligned}
\]
    
\end{example}

\begin{example}
    For $m=3$, $\ker\sigma_{n}$ is generated by ${\widetilde{\ker}\sigma_{2}}$, ${\widetilde{\ker}\sigma_{3}}$ and ${\widetilde{\ker}\sigma_{4}}$ with dimension $6,8,3$ respectively, where
    \[
    \begin{aligned}
        {\widetilde{\ker}\sigma_{2}}=span_{\mathbb{K}}& \{v_{i}\otimes v_{i},v_{i}\otimes v_{j}+y^2v_{j}\otimes v_{i}|1\le i< j\le 3\}
        \\{\widetilde{\ker}\sigma_{3}}=span_{\mathbb{K}}&\{\omega_{i}(v_{j}\otimes v_{j}),\omega _{s}(v_{i}\otimes v_{j}+y^2 v_{j}\otimes v_{i})\mid  1\le s\le i<j\le 3\}\cup \{y^2[v_{1},v_{2},v_{3}]+\\&(1+y^2+y^4)(v_{1}\otimes v_{3}\otimes v_{2}-v_{2}\otimes v_{3}\otimes v_{1})\}
        \\{\widetilde{\ker}\sigma_{4}}=span_{\mathbb{K}}&\{\omega_{1}\circ \omega_{2}(v_{3}\otimes v_{3}),\omega_{1}\circ \omega_{2}(v_{2}\otimes v_{3}+y^2 v_{3}\otimes v_{2}),\\& \omega_{1}(y^2[v_{1},v_{2},v_{3}]+(1+y^2+y^4)(v_{1}\otimes v_{3}\otimes v_{2}-v_{2}\otimes v_{3}\otimes v_{1}))\}
    \end{aligned}
    \]
\end{example}
\begin{remark}
Notice that when $m=3$,    $\widetilde{\ker}\sigma_{4}\subsetneqq \ker\sigma_{4}$.
\end{remark}

 Let $M_{l}=\mathbb{K}^{    \oplus l}$ with chosen basis, and $v_i$ acting on $M_{l}$ is represented by matrix multiplication from right by a matrix $A_i$ for $1\leq i \leq m$. When $A_iA_j=A_jA_i$ for all $1\leq i,j \leq m$, $M_{l}$ becomes a $V_m$-module and we denote it by $(M_l;A_1,...,A_m)$. As one application, we compute the Betti number of the one-term Yang-Baxter homology with coefficient in $(M_l;A_1,...,A_m)$.

\begin{example}
    Consider the one-term Yang-Baxter homology with coefficient in the $l$-dimensional $V_{m}$-module $(M_l;A_1,...,A_m)$, $H^{YB}_{n}(M_l;A_1,...,A_m)$. The dimension of $H^{YB}_{n}(M_l;A_1,...,A_m)$ is:
    \[
    l\cdot m^{n}-r_{1}M(n)-(r_{1}+r_{2})M(n-1)-\cdots -(r_{m-1}+r_{m})M(n-m+1)-r_{m}M(n-m).
    \]
    where $$\begin{aligned}
        &r_{k}:=\dim(M{[V]_{k}})\\
        &r_{1}:=\dim(MV)=rank\begin{pmatrix}
A_{1} \\
A_{2} \\
\cdots \\
A_{m}
\end{pmatrix}
    \end{aligned}$$
    with $M[V]_{k}:=R_{M}\otimes{\rm id}_{V^{\otimes k-1} }(M\otimes [V]_{k})$.
    
\end{example}
\begin{proof}
 The dimension of $H^{YB}_{n}(M;A_{1},\cdots,A_{m})$ is 
 \[
 l\cdot m^{n}-\dim({\rm im}\,\partial^{YB}_{n+1})-\dim({\rm im}\,\partial^{YB}_{n})
 \]
 and we have \[
   {\rm im}\,\partial^{YB}_{n+1}=MV\otimes \ker\sigma_{n}\oplus M[V]_{2}\otimes \ker\sigma_{n-1} \oplus \cdots \oplus M[V]_{n+1}.
   \]
   As a demonstration, we give an explicit computation.
\begin{example}
    For $l=3,m=3$, we have \[r_{1}=rank\begin{pmatrix} A_{1}\\ A_{2}\\A_{3}\end{pmatrix};\  r_{2}=rank\begin{pmatrix}
  -A_{2}&A_{1}  &0 \\
  -A_{3}& 0 &A_{1} \\
  0&-A_{3}  & A_{2}
  \end{pmatrix};\  r_{3}=rank\begin{pmatrix}
  A_{1}&  A_{2}&A_{3}
\end{pmatrix}.\]
Consider \[
  A_{1}=\begin{pmatrix}
  1& 0 &1 \\
 0 & 0 & 1\\
 0 & 0 &0
\end{pmatrix},\  A_{2}=\begin{pmatrix}
 1 &  0& 1\\
 0 & 0 &0 \\
0  &0  &0
\end{pmatrix},\ A_{3}=\begin{pmatrix}
1  & 0 &1 \\
0  &0  & 0\\
0  & 0 &0
\end{pmatrix}
  \] 
We have $r_{1}=2,\ r_{2}=4,\ r_{3}=2$. Thus, the dimension of $H_{n}(M=\mathbb{K}^{\oplus3};A_{1},A_{2},A_{3})$ is $$3\cdot3^{n}-2M(n)-(2+4)M(n-1)-(4+2)M(n-2)-2M(n-3)=3\cdot 3^{n}-2\cdot 3^{n}=3^{n}.$$
\end{example}
\end{proof}

\section*{Acknowledgements}
We would like to thank Zhiyun Cheng for valuable discussions. The first author was supported by the National Natural Science Foundation of China (Grant NO. 11971256 and 12471064). The second author was supported by the National Natural Science Foundation of China (Grant No. 11901229, 12371029, 22341304 and W2412041).

\include{YB/another}


\newpage
\begin{thebibliography}{0}

\bibitem{AI}   J. A. de Azcárraga, J. M. Izquierdo, n-ary algebras: a review with applications, 2010 J. Phys. A: Math. Theor. 43 293001
\bibitem{Bax} R.~J.~Baxter, Partition function of the Eight-Vertex lattice model
{\it Annals of Physics}, 70, 1, 1972, 193-228

\bibitem{Jon} V.~F.~R.~Jones, A polynomial invariant for knots via Non Neumann Algebras, Bull. Amer. Math. Soc. (N.S.) 12, 1985, 103–111.
\bibitem{Joy} D. Joyce, A classifying invariant of knots, the knot quandle, J. Pure Appl. Algebra 28
23(1) (1982) 37–65.

\bibitem {Leb-1}
V.~Lebed,
Braided objects: unifying algebraic structures and categorifying virtual
braids December 2012, Thesis (Ph.D.), Universit\'e Paris 7;

\bibitem{Prz-2} J.~H.~Przytycki, Knots and distributive homology: from arc colorings to Yang-Baxter homology,
 Chapter in: {\it New Ideas in Low Dimensional Topology}, World Scientific, Vol. 56, March-April 2015, 413-488.
e-print: \ {\tt arXiv:1409.7044 [math.GT]}

\bibitem{PW} J.~H.~Przytycki, X.~Wang, Equivalence of two definitions of set-theoretic Yang-Baxter homology and general Yang-Baxter homology, {\it Journal of Knot Theory and Its Ramiﬁcations} Vol. 27, No. 07, 1841013 (2018). 
\bibitem{PW2} J.~H.~Przytycki, X.~Wang, The second Yang–Baxter homology for the HOMFLYPT polynomial, {\it Journal of Knot Theory and Its Ramiﬁcations} Vol. 30, No. 13, 2141014 (2021).
\bibitem{RT} N. Reshetikhin, V. G. Turaev, Invariants of 3-manifolds via
link polynomials and quantum groups.
{\it Invent.
Math.}, 103,1991, 547–597.

\bibitem{Sol} A. Soloviev, Non-unitary set-theoretical solutions to the quantum Yang–Baxter equation, Math. Res. Lett. 7(5–6) (2000) 577–596.

\bibitem{Tur} V.~G.~Turaev, The Yang-Baxter equation and invariants of links, {\it Invent.
Math.}, 92,1988, 527-553.

\bibitem{Yan} C.~N.~Yang, Some Exact Results for the Many-Body Problem in one Dimension with Repulsive Delta-Function Interaction
{\it Phys. Rev. Lett.} 1967, 1312

\end{thebibliography}
\end{document}